\documentclass[pdflatex,sn-mathphys-ay]{sn-jnl}

\usepackage{graphicx}%
\usepackage{multirow}%
\usepackage{amsmath,amssymb,amsfonts}%
\usepackage{amsthm}%
\usepackage{mathrsfs}%
\usepackage[title]{appendix}%
\usepackage{xcolor}%
\usepackage{textcomp}%
\usepackage{manyfoot}%
\usepackage{booktabs}%
\usepackage{algorithm}%
\usepackage{algorithmicx}%
\usepackage{algpseudocode}%
\usepackage{listings}%

\theoremstyle{thmstyletone}%
\newtheorem{theorem}{Theorem}
\newtheorem{proposition}[theorem]{Proposition}%
\newtheorem{lemma}[theorem]{Lemma}%
\newtheorem{corollary}[theorem]{Corollary}%

\theoremstyle{thmstyletwo}%
\newtheorem{remark}{Remark}%

\theoremstyle{thmstylethree}%

\begin{document}

\title[On the probability of generating matrix incidence rings]{On the probability of generating matrix incidence rings}


\author*[1]{\fnm{N.A. Kolegov}}\email{na.kolegov@yandex.ru}

\affil*[1]{\orgname{Lomonosov Moscow State University},  \city{Moscow}, \postcode{119991}, \country{Russia}}


\abstract{The probability that a tuple of matrices together with all scalars generates a finite incidence ring is calculated. It is proved that all real and complex finite-dimensional incidence algebras are generated by two randomly chosen matrices.}




\maketitle
\section{Introduction }

Generators and relations of algebraic structures is the classical research field.  For~finite structures, a natural problem is to estimate the number of generating subsets. A more general approach, which can cover infinite structures in some cases, is to introduce an appropriate measure and estimate the probability of generating the algebraic structure.

A computer algebra approach of extracting generators and relations of matrix algebras was investigated by \citet*{Car06}. \citet*{AiKaMu13} showed that a real matrix $*$-algebra can be generated by four randomly chosen symmetric matrices. \citet*{KrMazPet12} developed methods to calculate probability of generating algebras which are free modules over orders in algebraic number fields. \citet*{SerSha24} proved that probability that two randomly-chosen elements generate a ``bounded'' finite associative algebra over a field converges to 1 as the cardinality of the algebra tends to infinity.

This paper studies generators of  matrix incidence rings over finite posets. Such rings consist of matrices with special patterns of zeros. For example, the rings of diagonal and upper-triangular matrices are particular examples of incidence rings. In general, for any finite poset and any base ring, one can construct the incidence ring. There are interesting applications in combinatorics, see the classical book by \citep{SpODon97}. Two finite posets are isomorphic iff their incidence algebras over a field are isomorphic~\citep[Theorem~7.2.2]{SpODon97}.  Incidence algebras are a rare case when all generating subsets can be described directly \citep{LonRos00}. Moreover, the theory of generators can be extended to matrices over noncommutative rings~\citep{Ko23}.

The main result of this paper is Theorem~\ref{number_of_gen}. It provides the precise formula for the number of matrix tuples of the given length that together with all scalar matrices generate an incidence ring of an arbitrary finite poset over an arbitrary finite base ring. As a consequence, the probability of generating a finite incidence ring by a tuple of matrices is determined. The last section considers a similar question for incidence algebras over the fields of real and complex numbers. In this case it is shown that the set of pairs of matrices that do not generate the algebra is an affine algebraic set. A natural probabilistic interpretation of this result shows that the algebra is generated by two random matrices.

This paper is organized as follows. In~Section~\ref{prelim}, necessary definitions and some preliminary results are given. In~Section~\ref{mat_r}, the number of matrix tuples of given length that together with all scalar matrices generate an incidence ring over a finite full matrix algebra is calculated (Lemma~\ref{mat_ring_formula}). In~Section~\ref{finite_r}, the previous result is generalized to an arbitrary finite incidence a ring (Theorem~\ref{number_of_gen}). In Section~\ref{measure}, a natural probability measure is introduced for matrix tuples on a sphere and it is shown that two randomly chosen matrices generate an incidence algebra over the fields of real and complex numbers (Corollary~\ref{prob_over_R_C}).

\section{Preliminaries}\label{prelim}

Throughout the paper, ${\mathcal R}$ denotes a nonzero associative ring with $1$, which can be noncommutative. Let $M_n({\mathcal R})$ be the ring of all $n\times n$ matrices over~${\mathcal R}$. Denote~by~$T_n({\mathcal R})$ and $D_n({\mathcal R})$ the subrings of $M_n({\mathcal R})$ consisting of upper-triangular and diagonal matrices, respectively. The $i,j$-th entry of a matrix $A\in M_n({\mathcal R})$ is denoted by $(A)_{ij}$, or $a_{ij}$. Let $I_n$, or $I$, be the identity matrix in $M_n({\mathcal R})$ and $E^{(n)}_{ij}$, or $E_{ij}$, be the $i,j$-th matrix unit in $M_n({\mathcal R})$. Also, ${\mathcal R} I_n =\{r\cdot I_n~|~r\in{\mathcal R}\}$ is the set of all scalar matrices in $M_n({\mathcal R})$.

Consider a partial order $\preceq$ on the set $\{1,\ldots, n\}$. For $i,j\in\{1,\ldots,n\}$, introduce the {\em covering} relation $i\prec: j$ which means that there is no element~$k$ satisfying $i\prec k \prec j$.

An {\it incidence ring} is defined as
\[{\mathcal A}={\mathcal A}_n(\preceq,{\mathcal R}) = \{A\in M_n({\mathcal R})~|~(A)_{ij} = 0~\text{for all pairs}~i\not\preceq j\}.\]
Then ${\mathcal A}_n(\preceq,{\mathcal R})$ is a finitely-generated free right (or left) ${\mathcal R}$-module with a basis of the matrix units $\{E_{ij}~|~i\preceq j\}$. Since the partial order $\preceq$ is a transitive relation, ${\mathcal A}$~is~closed under the matrix multiplication. Reflexivity of $\preceq$ implies that ${\mathcal A}$ contains $D_n({\mathcal R})$ as a subring. Also, any incidence ring can be naturally embedded into $T_n({\mathcal R})$ \citep[Proposition~1.2.7]{SpODon97} because $\preceq$ is antisymmetric. Note that $D_n({\mathcal R}) ={ {\mathcal A}_n(=,{\mathcal R})}$ and $T_n({\mathcal R}) = {\mathcal A}_n(\leq,{\mathcal R})$, where $\leq$ is the standard linear order. See \citep{KolMar19} for some discussion of functional and matrix languages for describing incidence rings.

The Jacobson radical of an incidence ring was calculated by \citet[Proposition~4]{Voss80}:
\begin{equation}\label{eqJ}
	J({\mathcal A}) = \sum_{i=1}^n E_{ii}J({\mathcal R}) + \sum_{i\prec j} E_{ij} {\mathcal R}.
\end{equation}

Let ${\mathcal T}$ be any associative ring with identity $1_{\mathcal T}$ and $S$ be its nonempty subset. Denote by $\langle{S}\rangle_{{\text{Ring}}}$ the subring generated by $S$, that is, the minimum subring of ${\mathcal T}$ that contains $S\cup\{1_{\mathcal T}\}$. If ${\mathcal T} = \langle{S}\rangle_{{\text{Ring}}}$, then $S$ is said to generate ${\mathcal T}$ as a ring. If ${\mathcal T}$ is an algebra over a commutative ring ${\mathcal K}$ and $\langle{S\cup {\mathcal K} 1_{\mathcal T}}\rangle_{{\text{Ring}}}=T$, then $S$ is said to generate ${\mathcal T}$ as a ${\mathcal K}$-algebra.

Consider an incidence ring ${\mathcal A} = {\mathcal A}_n(\preceq,{\mathcal R})$. Given a~natural number $m$, denote by ${\mathcal A}^m = {\mathcal A}\times\ldots\times{\mathcal A}$ the direct product of $m$ copies of the ring ${\mathcal A}$.
Introduce the~set
\[{\text{Gen}}_m({\mathcal A},{\mathcal R}) = \{(A_1,\ldots,A_m)\in{\mathcal A}^m~|~\{A_i\}_{i=1}^m\cup {\mathcal R} I~\text{generates}~{\mathcal A}~\text{as a ring}\}\]
following the similar notation in~\citep{KrMazPet12}. Then the minimal number of generators can be defined as
\[{\mathrm  {mgen}}({\mathcal A})=\min\{m\in{\mathbb N}~|~{\text{Gen}}_m({\mathcal A},{\mathcal R})\neq\varnothing\}.\]
Generators of incidence rings were described in \citep[Theorem]{LonRos00} and \citep[Theorem~1.1]{Ko23}. The approach based on ring theory was developed in \citep[Lemma~4]{Ker98}. The quantity ${\mathrm  {mgen}}({\mathcal A})$ was calculated in the case when the ring~${\mathcal R}$ is semilocal~\citep[Theorem~1.2]{Ko23}. In particular, this result is applicable to any finite ring ${\mathcal R}$.

Throughout the paper, ${\mathbb F}$ is an arbitrary field and ${\mathbb F}_q$ is a finite field of cardinality~$q$.  In order to calculate cardinality $|{\text{Gen}}_m({\mathcal A},M_k({\mathbb F}_q))|$  we need two simple lemmas from linear algebra. The proofs of them are given for the sake of completeness.

\begin{lemma}\label{operator_lemma}
	Let $\varphi: M_k({\mathbb F}) \longrightarrow M_k({\mathbb F})$ be an ${\mathbb F}$-linear mapping, $k\geq 1$. Then there exist such natural number $r$ and matrices $P_1,\ldots, P_r$, $Q_1,\ldots, Q_r$ that
	\[\varphi(X) = \sum_{i = 1}^r P_i X Q_i,\quad \forall X\in M_k({\mathbb F}).\]
\end{lemma}
\begin{proof}
	Consider a basis of the matrix units $\{E_{\alpha\beta}\}_{\alpha,\beta=1}^k$ in $M_k({\mathbb F})$. For all indexes $\alpha,\beta,\gamma,\delta \in \{1,\ldots,k\}$ denote by $\varepsilon_{\alpha,\beta,\gamma,\delta}$ the linear mapping that is defined on the~basis elements as
	\begin{equation*}
		\varepsilon_{\alpha,\beta,\gamma,\delta} (E_{ij}) = 
		\begin{cases}
			E_{\gamma\delta}  & \text{if $(i,j) = (\alpha,\beta)$,} \\
			O & \text{otherwise}.
		\end{cases}
	\end{equation*}
	We have $k^4$ operators $\varepsilon_{\alpha,\beta,\gamma,\delta}$, which are linearly independent over ${\mathbb F}$. Hence, they constitute a basis of the space of linear operators on $M_k({\mathbb F})$. At the same time, $\varepsilon_{\alpha,\beta,\gamma,\delta} (X) = E_{\gamma \alpha}X E_{\beta \delta}$ for all $X\in M_k({\mathbb F})$.
\end{proof}

\begin{lemma}\label{system_of_operators_lemma}
	
	Let $V$ be a linear space over a field ${\mathbb F}$. Consider arbitrary vectors $v_1,\ldots, v_m$ and $w_1,\ldots, w_m$ form $V$. Then the following two conditions are equivalent:
	
	1) There exist linear operators $\varphi_i: V\longrightarrow V$ for $i = 1,\ldots, m$ that satisfy	
	\begin{equation*}
		\begin{cases}
			\varphi_1(v_1) + \ldots + \varphi_m(v_m) = 0,\\
			
			\varphi_1(w_1) + \ldots + \varphi_m(w_m) \neq 0.
		\end{cases}
	\end{equation*}
	
	2) There is no such $\lambda\in{\mathbb F}$ that $w_i = \lambda v_i$ for all $i = 1,\ldots, m$.
	\end{lemma}
\begin{proof} 
	1) $\Rightarrow$ 2) The direct substitution $w_i = \lambda v_i$ into the system results in contradiction.
	
	2) $\Rightarrow$ 1) Consider the outer direct sum $V^m = V\oplus\ldots\oplus V$~$(m~\text{times})$. For~$i=1,\ldots,m$, we denote  the i-th direct summand by $V_i$, the projection by ${\pi_i:V^m{\longrightarrow} V_i}$, and the natural embedding by $\iota_i: V_i{\longrightarrow} V^m$.
	The space $V^m$ contains such vectors $v,w$ that their projections onto each $V_i$ equal $v_i$ and $w_i$, respectively. Item~2 implies that $w\neq \lambda v$ for any $\lambda\in{\mathbb F}$. It means that either the vectors $v,w$ are linearly independent, or $v = 0$, but $w\neq 0$. Anyway, there exists such linear operator $\Phi: V^m\longrightarrow V^m$ that $\Phi(w)=u\neq 0$ and $\Phi(v) = 0$. Since $u\neq 0$, it has at least one nonzero projection, say $\pi_{k}(u)\neq 0$ for some $k\in\{1,\ldots,m\}$. Then we set ${\varphi}_i = \pi_k\circ\Phi\circ \iota_i$. These operators indeed satisfy the system since
	\begin{align*}
		& \sum_{i = 1}^m {\varphi}_i(v_i) =  (\pi_k\circ \Phi )\left( \sum_{i = 1}^m  \iota_i(v_i)\right) = \pi_k(\Phi(v)) =\pi_k(0)= 0,\\
		&\sum_{i = 1}^m {\varphi}_i(w_i) =  (\pi_k\circ \Phi )\left( \sum_{i = 1}^m  \iota_i(w_i)\right) = \pi_k(\Phi(w))=\pi_k(u) \neq 0.
	\end{align*}
\end{proof}

\section{Incidence rings over finite matrix algebras}\label{mat_r}

The main result of this section is an expression for $|{\text{Gen}}_m({\mathcal A}, M_k({\mathbb F}_q))|$ in Lemma~\ref{mat_ring_formula}. Before this, we need another lemma, which is a modification of \citep[Theorem~1.1]{Ko23}.

\begin{lemma}\label{criterion_reformulation_over_matrix_ring}
	Let ${\mathcal R}\cong M_k({\mathbb F})$ be the ring of all $k\times k$ matrices over a field~${\mathbb F}$, ${k\geq 1}$. Consider an incidence ring ${\mathcal A} = {\mathcal A}_n(\preceq, {\mathcal R})$, $n\geq 2$ and any its non\-emp\-ty finite subset $S = \{A_1,\ldots, A_m\}$. Then $S\cup {\mathcal R} I_n$ generates ${\mathcal A}$ as ring if and only if 
	
	1) the $n\times m$ matrix over ${\mathcal R}$
	\[\Delta = \begin{pmatrix}
		(A_1)_{11}&\ldots &(A_m)_{11}\\
		\ldots &\ldots &\ldots\\
		(A_1)_{nn}&\ldots &(A_m)_{nn}
	\end{pmatrix} \] 
	has no identical rows;
	
	2) for each pair $i\prec : j$, the vectors
	\begin{align*}
		&v = ((A_1)_{ii} - (A_1)_{jj},~ (A_2)_{ii} - (A_2)_{jj}, ~\ldots,~ (A_m)_{ii} - (A_m)_{jj})\in{\mathcal R}^m,\\
		&w = ((A_1)_{ij}, ~(A_2)_{ij}, ~\ldots, ~(A_m)_{ij})\in{\mathcal R}^m
	\end{align*}

	are linearly independent \underline{over the field ${\mathbb F}$}. 
\end{lemma}
\begin{proof}
	Since the ring ${\mathcal R}$ is simple, its ideal is improper iff it contains a nonzero element. Then Item~$1$ of \citep[Theorem~1.1]{Ko23} can be rewritten as follows. For each pair $i\neq j$, there exists such $\alpha\in \{1,\ldots, m\}$ that $(A_\alpha)_{ii}\neq (A_\alpha)_{jj}$.  This is equivalent to the requiring that all rows of the matrix $\Delta$ are distinct.
	
	Next, consider a pair $i\prec : j$ and apply Item $2$ of \citep[Theorem~1.1]{Ko23}. Again simplicity of the ring ${\mathcal R}$ ensures that $\mathfrak{b}_{ij} = {\mathcal R}$ iff $\mathfrak{b}_{ij}\neq 0$. So the ring ${\mathcal R}$ must contain such elements $p_{\alpha,\beta}$, $q_{\alpha,\beta}$ for $\alpha = 1,\ldots,m$ and $\beta = 1,\ldots,r_\alpha$ that the matrix $B = \sum_{\alpha = 1}^m \sum_{\beta = 1}^{r_\alpha} p_{\alpha,\beta}\cdot  A_\alpha\cdot q_{\alpha,\beta}$ satisfies $(B)_{ii} = (B)_{jj}$ and $(B)_{ij}\neq 0$. In other words,
	\begin{equation*}
		\begin{cases}
			\sum\limits_{\beta = 1}^{r_1} p_{1,\beta}((A_1)_{ii} - (A_1)_{jj}) q_{1,\beta} + \ldots + \sum\limits_{\beta = 1}^{r_m} p_{m,\beta}((A_m)_{ii} - (A_m)_{jj}) q_{m,\beta} = 0,\\
			
			\sum\limits_{\beta= 1}^{r_1} p_{1,\beta}\cdot (A_1)_{ij}\cdot q_{1,\beta} + \ldots + \sum\limits_{\beta = 1}^{r_m} p_{m,\beta}\cdot(A_m)_{ij}\cdot q_{m,\beta} \neq 0.
		\end{cases}
	\end{equation*}
	By Lemma~\ref{operator_lemma}, this is equivalent to the requiring that there exist such ${\mathbb F}$-linear operators $\varphi_1,\ldots,\varphi_m$ on the space ${\mathcal R}$ that
	\begin{equation*}
		\begin{cases}
			\varphi_1((A_1)_{ii} - (A_1)_{jj}) + \ldots + \varphi_m((A_m)_{ii} - (A_m)_{jj}) = 0,\\
			
			\varphi_1((A_1)_{ij}) + \ldots + \varphi_m((A_m)_{ij}) \neq 0,
		\end{cases}
	\end{equation*}
	Note that the vector $v\neq 0$ by Item~1. It remains to apply Lemma~\ref{system_of_operators_lemma}.
\end{proof}

Next, we deduce a formula for the number of matrix tuples that together with all scalars generate an incidence ring over a full matrix ring.

\begin{lemma}\label{mat_ring_formula}
	Let ${\mathcal R} \cong M_k({\mathbb F}_q)$ be the ring of all $k\times k$ matrices over a~finite filed~${\mathbb F}_q$, $k\geq 1$. Consider an incidence ring ${\mathcal A} = {\mathcal A}_n(\preceq,{\mathcal R})$, $n\geq 1$.  Then for natural $m\geq {\mathrm  {mgen}}~{\mathcal A}$, we have
	\begin{equation*}
		|{\text{Gen}}_m({\mathcal A},{\mathcal R})|= 
		\dfrac{q^{k^2 m}! }{(q^{k^2 m} - n)!}\cdot (q^{k^2 m} - q)^c \cdot q^{{k^2 m}(\rho - n - c)}.
	\end{equation*}
	Here $\rho$ and $c$ are cardinalities of the partial order and the covering relation, respectively, that is
	\begin{equation*}
		\rho= |\preceq| = |\{(i,j)~|~i\preceq j\}|,\quad c = |\prec:| = |\{(i,j)~|~i\prec: j\}|.
	\end{equation*}
\end{lemma}
\begin{proof}
	If $n=1$, then $\rho = 1$, $c =0$ and $|{\text{Gen}}_m({\mathcal A},{\mathcal R})| = |{\mathcal R}|^m$, which is consistent with the formula. Let $n\geq 2$. We need to count matrix tuples $(A_1,\ldots, A_m)\in {\text{Gen}}_m({\mathcal A},{\mathcal R})$.
	
	[{\em Step 1.}] Consider the diagonal elements of the matrices $A_1,\ldots, A_m$. All~of them are contained in the matrix $\Delta$ from Lem\-ma~\ref{criterion_reformulation_over_matrix_ring}. All its rows must~be pairwise distinct by Item~1. So, the number of ways to choose the matrix $\Delta$ equals 
	\[|{\mathcal R}|^m(|{\mathcal R}|^m - 1)\cdot\ldots\cdot (|{\mathcal R}|^m -(n-1)) = \dfrac{|{\mathcal R}|^m!}{(|{\mathcal R}|^m - n)!} =\dfrac{q^{k^2 m}! }{(q^{k^2 m} - n)!}.\] Note~that $m\geq{\mathrm  {mgen}}~{\mathcal A}$ by our assumption and ${\mathrm  {mgen}}~{\mathcal A}\geq 	\lceil \log_{|{\mathcal R}|}n\rceil $ according to \citep[Theorem~1.2]{Ko23}. From this, at least one matrix $\Delta$ does exist.
	
	[{\em Step 2.}] Next, we treat the entries $(A_\alpha)_{ij}$ for $i\prec : j$ and $\alpha=1,\ldots,m$. According~to Item~2 of Lemma~\ref{criterion_reformulation_over_matrix_ring}, the vectors $v,w$ must be linearly independent over ${\mathbb F}$. The vector $v$ is uniquely determined by the matrix~$\Delta$. There are exactly $|{\mathcal R}|^m - q = q^{k^2 m} - q$ ways to choose the vector $w$ that  is not proportional to $v$. The~number of ways to choose the vectors~$w$ for all pairs $i\prec: j$ is equal to $(q^{k^2 m} - q)^c$.
	
	[{\em Step 3.}] It remains to consider the entries $(A_\alpha)_{ij}$ for $i\neq j$, $i\not\prec : j$ and $\alpha=1\ldots,m$. Lemma~\ref{criterion_reformulation_over_matrix_ring} ensures that all of them can be arbitrary. There~are ${m(\rho - n - c)}$ such entries. The number of ways to choose them equals $|{\mathcal R}|^{m(\rho - n - c)} = q^{{k^2 m}(\rho - n - c)}$.
	
	Multiplying the quantities obtained during Steps~1-3 we obtain the required expression for $|{\text{Gen}}_m({\mathcal A},{\mathcal R})|$.
\end{proof}

\begin{remark}
	When $k=m=1$ and $c=0$, the second multiplier $(q^{k^2 m} - q)^c$ becomes $0^0$, which should be interpreted as $1$. This follows from Step~2 of the proof.
\end{remark}

\section{Finite incidence rings}\label{finite_r}

In this section, we calculate the cardinality $|{\text{Gen}}_m({\mathcal A}, {\mathcal R})|$  for an arbitrary finite incidence ring ${\mathcal A}$ (Theorem~\ref{number_of_gen}). This requires two technical lemmas. The first one deals with the case when ${\mathcal R}$ is a direct product of rings.

\begin{lemma}\label{direct_prod_generators}
	Consider an incidence ring ${\mathcal A} = {\mathcal A}_n(\preceq,{\mathcal R})$, $n\geq 1$. Assume that ${\mathcal R}$ contains a set $\{e_1,\ldots,e_d\}$ of central pairwise orthogonal idempotents and $\sum_{i = 1}^d e_i  = 1_{\mathcal R}$. Then the following conditions are equivalent for a~nonempty subset $S{\subseteq} {\mathcal A}$:
	
	1) $S\cup {\mathcal R} I_n$ generates ${\mathcal A}$ as a ring;
	
	2) $e_i S\cup e_i{\mathcal R} I_{n}$ generates $e_i{\mathcal A}_i$ as a ring for each $i = 1,\ldots,d$. 
\end{lemma}
\begin{proof}
	1) $\Rightarrow$ 2) Since the set $S\cup{\mathcal R} I_n$ generates ${\mathcal A}$, its projection $e_i S\cup e_i {\mathcal R} I_n$ generates~$e_i{\mathcal A}$.
	
	2) $\Rightarrow$ 1) The set ${\mathcal R} I_n$ contains $\{e_1 I_n,\ldots, e_d I_n\}$. Hence, ${\langle{S\cup {\mathcal R} I_n}\rangle_{\text{Ring}}}\supseteq {\langle{e_i S\cup e_i{\mathcal R} I_n}\rangle_{\text{Ring}}} = e_i{\mathcal A}.$
	The equality $e_1{\mathcal A}+\ldots+e_d{\mathcal A} = {\mathcal A}$ implies that  ${\langle{S\cup {\mathcal R} I}\rangle_{\text{Ring}}} ={\mathcal A}$.
\end{proof}

The following lemma allows us to consider generators modulo the Jacobson radical of the ring ${\mathcal R}$.

\begin{lemma}\label{J_generators}
	Consider an incidence ring ${\mathcal A} = {\mathcal A}_n(\preceq,{\mathcal R})$, $n\geq 1$. Denote by $J = J({\mathcal R})$ the Jacobson radical of the ring ${\mathcal R}$. Let $\overline{{\mathcal R}} = {\mathcal R}/ J$ and $\overline{\mathcal A} = {{\mathcal A}_n(\preceq,\overline{\mathcal R})}$. Introduce the natural ring homomorphism $\varphi:{\mathcal A}{\longrightarrow}\overline{\mathcal A}$ given by $(\varphi(A))_{ij} = (A)_{ij}+J$ for all $i,j=1,\ldots,n.$ For a nonempty subset $S{\subseteq}{\mathcal A}$, denote $\overline{S}=\varphi(S)$. Then the following conditions are equivalent:
	
	1) $S\cup {\mathcal R} I_n$ generates ${\mathcal A}$ as a ring;
	
	2) $\overline{S}\cup\overline{{\mathcal R} I_n}$ generates $\overline{\mathcal A}$ as a ring.
\end{lemma}
\begin{proof}
	1) $\Rightarrow$ 2) It follows from the fact that $\varphi$ is a surjective homomorphism.
	
	2) $\Rightarrow$ 1) Denote by ${\mathcal B}$ the ring generated by the set $S\cup{\mathcal R} I_n$. By Item~2, $\varphi({\mathcal B}) = {\mathcal A}$ and so ${\mathcal B} + \sum_{i\preceq j} E_{ij} J = {\mathcal A}$. However, $\sum_{i\preceq j} E_{ij} J {\subseteq} J({\mathcal A})$ according to Equation~\eqref{eqJ}. Hence, ${\mathcal B} + {\mathcal A} J({\mathcal A}) = {\mathcal A}$. At the same time, ${\mathcal A}$ is finitely-generated by $\{E_{ij}\}_{i\preceq j}$ as a right ${\mathcal R}$-module. The noncommutative version of Nakayama's lemma implies that ${\mathcal B} = {\mathcal A}$, see \citep[Sec.~4.2 for $R={\mathbb Z}$ ]{Pierce82}.
\end{proof}

Now we are ready to prove the main theorem of the paper.

\begin{theorem}\label{number_of_gen}
	Let ${\mathcal R}$ be a finite ring. Consider the Wedderburn--Artin decomposition modulo the Jacobson radical ${\mathcal R}/J({\mathcal R})\cong M_{n_1}({\mathbb F}_{q_1})\times\ldots \times M_{n_d}({\mathbb F}_{q_d})$, where each ${\mathbb F}_{q_i}$ is the finite field of cardinality $q_i$, $n_i\geq 1$. Let ${\mathcal A} = {\mathcal A}_n(\preceq,{\mathcal R})$ be the incidence ring, $n\geq 1$. Then for any natural $m\geq {\mathrm  {mgen}}({\mathcal A})$, we have
	\begin{equation*}
		|{\text{Gen}}_m({\mathcal A},{\mathcal R})|= |J({\mathcal R})|^{m\rho}\cdot\prod\limits_{i = 1}^d \frac{{q_i}^{n_i^2 m}! }{(q_i^{n_i^2 m} - n)!}\cdot (q_i^{n_i^2 m} - q_i)^c \cdot q_i^{{{n_i}^2 m}(\rho - n - c)}.
	\end{equation*}
	Here $\rho$ and $c$ are cardinalities of the partial order and the covering relation, respectively, that is
	\begin{equation*}
		\rho = |\preceq| = |\{(i,j)~|~i\preceq j\}|,\quad c = |\prec:| = |\{(i,j)~|~i\prec: j\}|.
	\end{equation*}
	Note that the formula uses the convention $0^0=1$.
\end{theorem}
\begin{proof}
	In terms of Lemma~\ref{J_generators}, given any $\overline{A}\in\overline{{\mathcal A}}$ then 
	\[|\varphi^{-1}(\overline{A})| =\prod_{i\preceq j} | (A)_{ij} + J | = |J|^\rho.\]
	For a tuple $(\overline{A}_1,\ldots,\overline{A}_m)\in\overline{\mathcal A}^m$, we have $|\varphi^{-1}(\overline{A}_1)\times\ldots\times \varphi^{-1}(\overline{A}_m)| =  |J|^{m\rho}$. It follows from Lemma~\ref{J_generators} that
	\[|{\text{Gen}}_m({\mathcal A},{\mathcal R}) | = |J|^{m\rho} \cdot |{\text{Gen}}_m(\overline{\mathcal A},\overline{\mathcal R})|. \]
	
	Next, let $\overline{e}_1,\ldots,\overline{e}_d$ be such set of orthogonal central idempotents in $\overline{\mathcal R}$ that $\sum_i \overline{e}_i = 1_{\overline{{\mathcal R}}}$ and $\overline{e}_i\overline{\mathcal R}\cong M_{n_i}({\mathbb F}_{q_i})$. Note that $\overline{e}_i\overline{\mathcal A} = {\mathcal A}_n(\preceq, \overline{e}_i\overline{\mathcal R})$. Applying Lemma~\ref{direct_prod_generators} to the ring $
	\overline{{\mathcal A}}={\mathcal A}_n(\preceq,\overline{\mathcal R})$ we obtain
	\[	|{\text{Gen}}_m(\overline{\mathcal A},\overline{\mathcal R})| = 	|{\text{Gen}}_m(\overline{e}_1 \overline{\mathcal A},\overline{e}_1\overline{\mathcal R})| \cdot\ldots\cdot 	|{\text{Gen}}_m(\overline{e}_d \overline{\mathcal A},\overline{e}_d\overline{\mathcal R})|.\]
	The expression for $|{\text{Gen}}_m(\overline{e}_i \overline{\mathcal A},\overline{e}_i\overline{\mathcal R})|$ is given by Lemma~\ref{mat_ring_formula}. 
\end{proof}

The previous theorem can be reformulated in terms of the standard discrete probability measure on ${\mathcal A}^m$.

\begin{corollary}
	In notations of Theorem~\ref{number_of_gen} the probability that a tuple of $m$~matrices $(m\geq{\mathrm  {mgen}}({\mathcal A}))$ together with all scalar matrices generate the incidence ring is equal~to
	\[P = \prod\limits_{i = 1}^d  \left(1 - \dfrac{1}{q_i^{n_i^2 m-1}}\right)^{c} ~ \prod\limits_{\ell = 1}^{n-1} \left(1 - \dfrac{\ell}{q_i^{n_i^2 m}}\right).\]
\end{corollary}
\begin{proof}
	The cardinality of the direct product ${\mathcal A}^m$ can be represented as
	\[|{\mathcal A}^m| = |{\mathcal R}|^{m\rho} = |J({\mathcal R})|^{m\rho}\cdot |{\mathcal R}/ J({\mathcal R})|^{m\rho} = |J({\mathcal R})|^{m\rho}\cdot \prod_{i = 1}^d q_i^{n^2_im\rho}.\]
	It follows from the theorem that the required probability is equal to
	\[P = \dfrac{|{\text{Gen}}_m({\mathcal A},{\mathcal R})|}{|{\mathcal A}^m|}= \prod\limits_{i = 1}^d \frac{{q_i}^{n_i^2 m} \cdot\ldots\cdot ({q_i}^{n_i^2 m} - (n - 1))(q_i^{n_i^2 m} - q_i)^c}{ q_i^{{{n_i}^2 m}(n + c)}},\]
	which coincide with the expression from the corollary.
\end{proof}

\section{Incidence algebras over the fields of real and complex numbers}\label{measure}

Now we consider generators of an incidence algebra ${\mathcal A} =  {\mathcal A}_n(\preceq,{\mathbb R})$ over the field ${\mathbb R}$ of real numbers. In this case, the direct sum  ${\mathcal A}^m$ can be viewed as the arithmetic vector space ${\mathbb R}^{m\rho}$, $\rho=\dim_{\mathbb R}{\mathcal A}$ with the standard Euclidean norm. Given a tuple of matrices $(A_1,\ldots,A_m)\in{\mathcal A}^m$, then ${||(A_1,\ldots,A_m)||} = {\sqrt{\sum_{k = 1}^m\sum_{i\preceq j} (A_k)^2_{ij}}}.$ A~nonzero tuple $(A_1,\ldots,A_m)$ generates the algebra ${\mathcal A}$ iff $(\lambda A_1,\ldots, \lambda A_m)$ generates ${\mathcal A}$ for any $\lambda\in{\mathbb R}\setminus\{0\}$. We may choose $\lambda = \dfrac{1}{||(A_1,\ldots,A_m)||}$ and consider matrix tuples on the unit sphere ${\bf S}{\subseteq} {\mathbb R}^{m\rho}$ of dimension~$m\rho - 1$. In other words, the intersection ${\bf S}\cap {\text{Gen}}_m({\mathcal A},{\mathbb R})$ will be considered.

There exists only one uniformly distributed Borel regular probability measure on the Euclidean sphere \citep[Theorem 3.4]{Mat95}. It can be constructed as follows. Let $\Pi{\subseteq}{\bf S}$ be a subset. Consider the cone with the base $\Pi$ and the zero apex
\[{\mathcal K}(\Pi) = \bigcup\limits_{(A_1,\ldots,A_m)\in\Pi} \{(tA_1,\ldots,tA_m)~|~t\in[0,1]\}.\]
The space ${\mathbb R}^{m\rho}$ is equipped with the Lebesgue measure $\mu$. Then the probability measure $P$ on ${\bf S}$ can be defined by
\[P(\Pi) = \dfrac{1}{\mu({\mathcal K}({\bf S}))}\cdot \mu({\mathcal K}(\Pi)) = \dfrac{\Gamma(\frac{m\rho}{2}+1)}{\pi^{\frac{m\rho}{2}}}\cdot \mu({\mathcal K}(\Pi)),\]
where $\mu({\mathcal K}({\bf S}))=\mu({\bf B})$ is the volume of $m\rho$-dimensional unit ball ${\bf B}$ and $\Gamma$~is the Euler's gamma function. So a subset of the sphere is considered to be measurable iff its cone is Lebesgue-measurable.

A similar construction is possible for the algebras over the field of complex numbers ${\mathbb C}$. In this case ${\mathcal A}^{m}$ is identified with ${\mathbb R}^{2m\rho}$ and the probability measure on the sphere ${\bf S}$  of dimension~$2m\rho-1$ is introduced in the same way.

\begin{proposition}\label{algebraic_set}
	Consider an incidence algebra ${\mathcal A} =  {\mathcal A}_n(\preceq,{\mathbb F})$ over any field~${\mathbb F}$. Let~$m\geq 2$ be a natural number. Then ${\text{Gen}}_m({\mathcal A},{\mathbb F})\neq\varnothing$ and the complement ${\mathcal A}^m\setminus {\text{Gen}}_m({\mathcal A},{\mathbb F})$ is an affine algebraic set.
\end{proposition}
\begin{proof}
	We have ${\text{Gen}}_m({\mathcal A},{\mathbb F})\neq\varnothing$ because ${\mathrm  {mgen}}({\mathcal A})\leq2$ for incidence algebras over infinite fields \citep[Corollary~4.7]{KolMar19}. For each pair $i,j\in\{1,\ldots,n\}$ with $i\neq j$, introduce the subsets of ${\mathcal A}^m$:
	\begin{align*}
		\Phi_{ij} = \{(A_1,\ldots,A_m)\in {\mathcal A}^m~|&~(A_\alpha)_{ii} = (A_\alpha)_{jj,}~\alpha = 1,\ldots,m\},\\
		\Psi_{ij} = \{(A_1,\ldots,A_m)\in {\mathcal A}^m~|&~\exists(\lambda,\mu)\in{\mathbb F}\times{\mathbb F}\setminus\{(0,0)\}:
		\\\lambda(A_\alpha)_{ij} &+ \mu ((A_\alpha)_{ii} - (A_\alpha)_{jj})=0,~\alpha = 1,\ldots,m\}.
	\end{align*}
	The negation of Items~1,2 of Lemma~\ref{criterion_reformulation_over_matrix_ring} for $k = 1$ implies that
	\[{\mathcal A}^m\setminus {\text{Gen}}_m({\mathcal A},{\mathbb F})=\bigcup~\{\Phi_{ij}~|~i\neq j\} \cup \bigcup~\{\Psi_{ij}~|~i\prec: j\}.\]
	The matrix elements $x_{\alpha ij} = (A_\alpha)_{ij}$ for $\alpha=1,\ldots,m$ and $i\preceq j$ are coordinates in ${\mathcal A}^m\cong {\mathbb F}^{m\rho}$. Then $\Phi_{ij}$ is the set of solutions of the system $x_{\alpha ii} - x_{\alpha jj} =0$, $\alpha = 1,\ldots,m$.
	
	Consider the set $\Psi_{ij}$. It consists of such matrix tuples $(A_1,\ldots,A_m)$ that the matrix
	\[\begin{pmatrix}
		(A_1)_{ii} - (A_1)_{jj} & (A_2)_{ii} - (A_2)_{jj} & \ldots & (A_m)_{ii} - (A_m)_{jj} \\
		(A_1)_{ij} & (A_2)_{ij} & \ldots & (A_m)_{ij} 
	\end{pmatrix}\]
	has rank $\leq 1$. This is equivalent to the requiring that the minors
	\[\begin{vmatrix}
		(A_\alpha)_{ii} - (A_\alpha)_{jj} &  {(A_{\beta})_{ii} - (A_{\beta})_{jj}}\\
		(A_\alpha)_{ij}  & 	(A_{\beta})_{ij} 
	\end{vmatrix} =0\]
	vanish for all $\alpha,\beta\in\{1,\ldots,m\}$. Hence, $\Psi_{ij}$ is also the set of solutions of a system of polynomial equations. Thus, the complement ${\mathcal A}^m\setminus {\text{Gen}}_m({\mathcal A},{\mathbb F})$  is the finite union of affine algebraic sets and so it is affine algebraic as well.
\end{proof}

\begin{corollary}\label{prob_over_R_C}
	Consider an incidence algebra  ${\mathcal A} =  {\mathcal A}_n(\preceq,{\mathbb F})$ over the~fields of real ${\mathbb F} = {\mathbb R}$ or complex numbers ${\mathbb F} = {\mathbb C}$. Then two matrices generate ${\mathcal A}$  with probability 1.
\end{corollary}
\begin{proof}
	Denote $V = 
	{\mathcal A}^2\setminus {\text{Gen}}_2({\mathcal A},{\mathbb F})$. By Proposition~\ref{algebraic_set}, $\mu(V) = 0$ since $V$~is the set of common zeros of polynomials, which are analytic functions. 
	
	Given any $v\in V$ and $\lambda \in {\mathbb{F}}$, then ${\lambda\cdot v}\in V$ by the definition of ${\text{Gen}}_2({\mathcal A},{\mathbb F})$. So the cone ${\mathcal K}( V \cap {\bf S})$ is a subset of~$V$. Then $\mu({\bf B})\cdot P({\bf S}\cap V)=\mu({\mathcal K}({\bf S}\cap V)) \leq \mu(V) = 0$.
\end{proof}


\bmhead{Acknowledgements}
The author is grateful to his scientific supervisor Professor O.V.~Markova for valuable discussions during the preparation of the paper. The~work was financially supported by Theoretical Physics and Mathematics Advancement Foundation ``BASIS'', grant 22-8-3-21-1.

\section*{Statements and Declarations}

\bmhead{Conflict of interest}
The author declares that he has no conflict of interest.


\begin{thebibliography}{1}
	
	\bibitem[Aiura et al. (2013) Aiura, Kakimura, Murota]{AiKaMu13}
	Aiura, D., Kakimura, N., Murota, K.: On the number of matrices to generate a matrix $*$-algebra over the real field. Linear Algebra Appl. 438, 1252--1266 (2013). https://doi.org/10.1016/j.laa.2012.08.022
	
	\bibitem[Carlson, Matthews (2006) Carlson, Matthews]{Car06}
	Carlson, J.F., Matthews, G.: Generators and relations for matrix algebras. J.~Algebra 300, 134--159 (2006). https://doi.org/10.1016/j.jalgebra.2006.01.043
	
	\bibitem[Kelarev et al. (1998) Kelarev, van der Merwe, van Wyk]{Ker98}
	Kelarev, A.V., van der Merwe, A.B., van Wyk, L.: The minimum number of idempotent generators of an upper triangular matrix algebra. J. Algebra 205, 605--616 (1998). https://doi.org/10.1006/jabr.1997.7405
	
	\bibitem[Kolegov (2023)]{Ko23}
	Kolegov, N.A.: On generators of incidence rings over finite posets. J. Algebra 619, 686--706 (2023). https://doi.org/10.1016/j.jalgebra.2022.12.025
	
	\bibitem[Kolegov, Markova (2019) Kolegov, Markova]{KolMar19}
	Kolegov, N.A., Markova, O.V.: Systems of generators of matrix incidence algebras over finite fields. J. Math. Sci., New York 240, 783--798 (2019). https://doi.org/10.1007/s10958-019-04396-6
	
	\bibitem[Kravchenko et al. (2012) Kravchenko, Mazur, Petrenko]{KrMazPet12}
	Kravchenko, R.V., Mazur, M., Petrenko, B.V.: On the smallest number of generators and the probability of generating an algebra. Algebra Number Theory, 6(2), 243--291 (2012). https://doi.org/10.2140/ant.2012.6.243
	
	\bibitem[Longstaff, Rosenthal (2000) Longstaff, Rosenthal]{LonRos00}
	Longstaff, W.E., Rosenthal, P.: Generators of matrix incidence algebras. Australas. J. Comb. 22 117--121 (2000). \\https://ajc.maths.uq.edu.au/pdf/22/ocr-ajc-v22-p117.pdf
	
	\bibitem[Mattila (1995) Mattila]{Mat95}
	Mattila, P.: Geometry of sets and measures in Euclidean spaces. Cambridge university press (1995). https://doi.org/10.1017/CBO9780511623813
	
	\bibitem[Pierce (1982) Pierce]{Pierce82}
	Pierce, R.S.: Associative algebras, Springer New York, NY, (1982). https://doi.org/10.1007/978-1-4757-0163-0
	
	
	\bibitem[Sercombe, Shalev (2024) Sercombe, Shalev]{SerSha24}
	Sercombe, D., Shalev, A.: Random generation of associative algebras. J. Lond. Math. Soc. 109(1), e12827 (2024). https://doi.org/10.1112/jlms.12827
	
	\bibitem[Spiegel, O'Donnell (1997) Spiegel, O'Donnell]{SpODon97}
	Spiegel, E., O'Donnell, C.J.: Incidence algebras. Marcel Dekker, Inc. (1997). https://doi.org/10.1201/9780203751176
	
	\bibitem[Voss (1980) Voss]{Voss80}
	Voss, E.R.: On the isomorphism problem for incidence rings. Illinois J. Math. 24(4), 624--638 (1980). https://doi.org/10.1215/ijm/1256047478
	
\end{thebibliography}
\end{document}